\documentclass{article}

\usepackage{amsmath,amsfonts,amssymb,amsthm}
\usepackage{hyperref}
\usepackage{fullpage}
\usepackage{color}
\usepackage{tikz}
\usepackage{braids}

\newcommand{\DOI}[1]{\href{http://dx.doi.org/#1}{DOI:#1}}
\newcommand{\arXiv}[1]{\href{http://arxiv.org/abs/#1}{arXiv:#1}}

\newcommand{\inv}{ l}

\newcommand{\mbfx}{\mathbf{x}}

\newtheorem{theorem}{Theorem}[section]
\newtheorem{proposition}[theorem]{Proposition}
\newtheorem{lemma}[theorem]{Lemma}

\theoremstyle{definition}
\newtheorem{definition}[theorem]{Definition}
\newtheorem{remark}[theorem]{Remark}

\begin{document}

\title{Probability distributions of multi--species $q$--TAZRP and ASEP as double cosets of parabolic subgroups}

\author{Jeffrey Kuan}

\date{\today}

\maketitle

\abstract{We write explicit contour integral formulas for probability distributions of the multi--species $q$--TAZRP and the multi--species ASEP starting with $q$--exchangeable initial conditions. The formulas are equal to the corresponding explicit contour integral formulas for the single--species $q$--TAZRP ([Korhonen--Lee 2014, Wang--Waugh 2016]) and ASEP [Tracy--Widom 2007], with a factor in front of the integral.

For the multi--species $q$--TAZRP, we use a decomposition theorem for elements of double cosets of parabolic subgroups in a Coxeter group. The set of distinguished double coset representatives with minimal length is viewed as a particle configuration. For the multi--species ASEP we use a more direct proof.  
}

\section{Introduction}
ASEP (asymmetric simple exclusion process, introduced in \cite{Sp}) and $q$--TAZRP (totally asymmetric zero range process, introduced in \cite{SW}) are examples of integrable models for which exact formulas for the transition probabilities can be written. When considering the models on the infinite line, these formulas are expressed as explicit $N$--fold contour integrals, where $N$ is the number of particles in the system, and were found in \cite{TW} (for ASEP) and \cite{KL} (for $q$--TAZRP; see also \cite{WW} for the inhomogeneous case) using Bethe Ansatz methods.

There are multi--species (also called multi--class) generalizations of ASEP and $q$--TAZRP, introduced in \cite{L} and \cite{T}, respectively. In these these multi--species models, there are $n$ species of particles, such that the projection onto the first $k$ species is Markov. In \cite{TW2}, the authors consider multi--species ASEP and prove explicit contour integral formulas for the location of the second class particle when the initial condition consists of a single second class particle located at $0$ and first class particles located at $\{1,2,\ldots\}$. In subsequent work \cite{TW3}, the same authors prove contour integral formulas for the transition probabilities for any initial condition and any number of species, assuming the number of particles is finite. However, the integrand is not explicitly written, except in a few cases, but is instead defined as a solution to certain consistency relations which are written in terms of the braid relations. 

In this paper, we consider multi--species ASEP and $q$--TAZRP with $q$--exchangeable initial conditions. These are initial conditions in which switching two nearest--neighbor particles of different species multiplies the probability by a factor of $q$. We find explicit contour integral formulas for the probability distributions, which are equal to the corresponding formulas for the single--species models, with a multiplicative factor in front of the integral. 

To prove the result for the multi--species $q$--TAZRP, we use a result from the theory of Coxeter groups, which uniquely decomposes elements of double cosets of parabolic subgroups. Roughly speaking, the left cosets correspond to allowing more than one particle to occupy a site, and the right cosets correspond to having more than one species of particles. The decomposition preserves the length function, and the dynamics can also be written in terms of the length function. For the multi--species ASEP, only one particle may occupy a site, so a more direct proof from Markov process generalities is used.

In section \ref{1}, we state and prove some lemmas involving Coxeter groups. Section \ref{2} shows the result for multi--species $q$--TAZRP, and section \ref{3} shows the result for multi--species ASEP.

\textbf{Acknowledgements.} The author was supported by NSF grant DMS--1502665 and the Minerva Foundation. The author would like to thank Alexei Borodin and Ivan Corwin for helpful discussions.

\section{Background and preliminary lemms}\label{1}

\subsection{Coxeter groups}
We recall some results about finite Coxeter groups; see e.g. \cite{BOOK}. A finite \textit{Coxeter group} is a group $W$ with a presentation
$$
W= \langle s \in S: s^2 = e \text{ for all }s\in S \text{ and } (s_is_j)^{m(s_i,s_j)}=e \text{ for all } s_i,s_j \in S \rangle
$$
where $m(s_i,s_j)$ is the order of $s_is_j$. 

If $S$ has $l$ elements, there is a representation $\tau$ of $W$ onto an $l$--dimensional vector space $V$.  Let $\alpha_1,\ldots,\alpha_l$ be a basis of $V$ and define a bilinear form on $V$ by 
$$
\langle \alpha_i,\alpha_j \rangle = -\cos \frac{\pi}{m_{ij}}.
$$
The map $\tau_{\sigma_i}$ is defined by 
$$
\tau_{\sigma_i}(v) = v - 2\langle \alpha_i,  v\rangle \alpha_i.
$$
This is the reflection across the hyperplane perpendicular to $\alpha_i$. The set $\Delta = \{\alpha_1, \ldots, \alpha_l\}$ is called the set of simple roots, and $\Phi=W(\Delta)$ is the set of all roots. Each root $\alpha \in \Phi$ has the form $\alpha = \sum_{i=1}^l \lambda_i \alpha_i$ where either all $\lambda_i \geq 0$ or all $\lambda_i \leq 0$. Define the set $\Phi^+$ of all positive roots to be the $\alpha\in \Phi$ for which all $\lambda_i \geq 0$. 

The length function $l(w)$ is the minimal length of an expression of $w$ as a product of generators $s_i \in S$.

\begin{proposition}\label{Three} 
(a) The only positive root made negative by $\tau_{s_i}$ is $\alpha_i$.

(b) For any $w$, the length $l(w)$ is the number of positive roots made negative by $w$.

(c) For any $w$, the $l(ws_i)$ is either $l(w) +1$ or $l(w)-1$. Similarly, $l(s_iw)$ is either $l(w)+1$ or $l(w)-1$.
\end{proposition}
\begin{proof}
Parts (a),(b),(c) follow from Propositions 2.2.6, 2.2.7 and 2.2.8 of \cite{BOOK}, respectively.
\end{proof}

Given a subset $J\subseteq S$, let $W_J$ be the subgroup of $S$ generated by $J$. Subgroups of this type are called parabolic subgroups, and are themselves Coxeter groups. Let $\Delta_J \subseteq \Delta$ be the set of simple roots $\alpha_j$ such that $s_j \in J$. 

\begin{proposition}\label{2.3.3}
Fix a parabolic subgroup $W_J$ in $W$. Then:

(a) Every left coset of $W_J$ has a unique representative with the fewest number of inversions in that coset.

(b) Let $D_J$ denote the set of distinguished coset representative from part (a). Every $w \in W$ has a unique decomposition $w = w^0 \overline{w}$ where $\overline{w } \in W_J$ and $\sigma^0 \in D_J$, which satisfies $\inv(w) = \inv(w^0) + \inv(\overline{w})$. 

(c) The set $D_J$ can be described by $\{ w \in W: \tau_w(\Delta_J) \subseteq \Phi^+\}$.

(d) Every right coset of $W_J$ has a unique representative with the fewest number of inversions in that coset, and $D_J^{-1}$ is the set of these coset representatives. Furthermore, there is a unique decomposition $w = \overline{w} w_0$ where $\overline{w} \in W_J$ and $\sigma_0 \in D_J^{-1}$ which satisfies $\inv(w) = \inv(\overline{w}) + \inv(w_0)$.
\end{proposition}
\begin{proof}
This is Proposition 2.3.3 of \cite{BOOK}.
\end{proof}

We will also need some results about \textit{double} cosets. Let $W_K$ be another parabolic subgroup of $W$, and define $D_{J,K} = D_{J}^{-1} \cap D_{K}$. Also see Corollary 2.8 of \cite{BKPST}, which references Proposition 8.3 of \cite{GS83} and Theorem (1.2) of \cite{Curtis85}, for similar statements.

\begin{proposition}\label{Cox}
(a) Each double coset $W_Jw W_K$ contains a unique element of $D_{J,K}$, and every $w \in D_{J,K}$ is the unique element of minimal inversions in its double coset $W_J w W_K$. 

(b) Let $w \in D_{J,K}$ and let $L$ be the parabolic subgroup defined by $\Delta_L = \Delta_{J} \cap w(\Delta_K)$. Every element of the double coset $W_J w W_K$ is uniquely expressible in the form $aw b$, where $a \in W_J \cap D_L$ and $b \in W_K$. Furthermore, $\inv(aw b) = \inv(a) + \inv(w) + \inv(b)$.

(c) The elements $a$ and $b$ can be constructed as follows. For $xw y$, let $x=ax'$ be the decomposition arising from $W_J = (W_J\cap D_L)L$, and let $b=w^{-1}x' w y$.

(d) If $w_0 \in D_J^{-1}$, then in the decomposition $w_0 = aw b$ in (b), the element $a$ is the identity. 

(e) If $w^0 \in D_K$, then in the decomposition $w^0 = aw b$ in (b), the element $b$ is the identity.
\end{proposition}
\begin{proof}
Part (a) is Proposition 2.7.3 of \cite{BOOK}.

Part (b) is Proposition 2.7.5 of \cite{BOOK}.

Part (c) follows from the first paragraph of the proof of Proposition 2.7.5 in \cite{BOOK}.

In part (d), there is a unique $w \in D_{J,K}$ such that $w_0 \in W_Jw W_K$, so take $w_0=aw b$ as in (c). Since $w_0 \in D_J^{-1}$, it is the unique element of its right coset $W_Jw_0$ with the fewest number of inversions. But the right coset $W_Jw_0$ equals $W_Jaw b = W_Jw b$, which implies that $\inv(w_0) \leq \inv(w b)$ . Since $w \in D_{J,K} \subseteq D_K$, then by the previous proposition $\inv(w b) = \inv(w) + \inv(b)$. Therefore $\inv(w_0) = \inv(a) + \inv(w) + \inv(b) \leq \inv(w) + \inv(b)$, so $\inv(a)=0$, implying that $a$ is the identity.

The proof of (e) is identical to the proof of (d).
\end{proof}

Note that the previous proposition is not true if $a$ is only required to be an element of $W_J$.

Part (a) implies that $\left| D_{J,K} \right| = \left| W_J\backslash W/ W_K \right|$. By the Cauchy--Frobenius lemma,

$$
\left| W_J \backslash W / W_K \right| = \frac{1}{ \vert W_J \vert \vert W_K\vert} \sum_{(x,y) \in W_J \times W_K} \left| W^{(x,y)}\right|
$$
where $W^{(x,y)}$ is the set of $w \in W$ such that $x w y^{-1} = w$. Because $W^{(e,e)}=W$, we must have $\vert D_{J,K}\vert \geq \vert W\vert / ( \vert W_J \vert \vert W_K\vert)$. In particular, the map from (d)
\begin{align*}
D_J^{-1} & \rightarrow D_{J,K} \times W_K \\
w_0 & \mapsto (w,b)
\end{align*}
is an injection but in general not a surjection. However, we do have:
\begin{lemma}\label{Moo}
The decomposition from (e) defines a map
\begin{align*}
D_K & \rightarrow \displaystyle \coprod_{w \in D_{J,K}} (W_J \cap D_{L})  \times \{w\}\\
w^0 & \mapsto (a,w),
\end{align*}
where $L$ depends on $w$, which is a bijection satisfying $l(w^0)=l(a) + l(w)$.
\end{lemma}
\begin{proof}
It is injective due to the uniqueness property in (b). To show it is surjective, it suffices to construct an inverse. In other words, we want to show that $xw \in D_K$ for any $w \in D_{J,K}$ and $x \in W_J\cap D_L$. Indeed, by part (c), one takes the unique decomposition $x = ax'$ from $W_J = (W_J \cap D_L)W_L$, but we must have $a=x$ because $x \in W_J\cap D_L$. Therefore $xw = aw = aw b$ for a unique $b\in W_K$, but clearly this means that $b=e$. 
\end{proof}

The next lemma will be useful at a later conjecture, but since it is true for Coxeter groups in general we state it here.
\begin{lemma}\label{PREV}
Suppose that $\sigma \in D_{J,K}$, and $l(s_i\sigma) = l(\sigma s_j) = l(\sigma)-1$. Then $s_i\sigma$ and $\sigma s_j$ are both in $D_{J,K}$.
\end{lemma}
\begin{proof}
Suppose that $\sigma s_j \notin D_{J,K}$. Then there exist $x\in W_J$ and $y\in W_K$ such that $l(x\sigma s_j y) < l(\sigma s_j)$. This implies that $l(x\sigma y) \leq l(x\sigma s_j y) + 1 \leq l(\sigma s_j) < l(\sigma)$, which contradicts the assumption that $\sigma\in D_{J,K}$. The proof for $s_i\sigma$ is identical.
\end{proof}

When the Coxeter group $W$ is the symmetric group $S(N)$ and the generators $S=\{s_1,\ldots,s_{N-1}\}$ are the transpositions $s_i = (i \ \ i+1)$, the parabolic subgroups are called Young subgroups. We will write Young subgroups as $H=S(m_1) \times S(m_2) \times \cdots \times S(m_r)$, where $S(m_1)$ acts on $\{1,\ldots,m_1\}$, $S(m_2)$ acts on $\{m_1+1,\ldots,m_1+m_2\}$, and so on. The positive integers $m_1,\ldots,m_r$ will be assumed to sum to $N$, so terms $S(1)$ are not excluded. Given a sequence of integers $\mathbf{m}=(m_1,\ldots,m_r)$, let $S(\mathbf{m})$ denote the Young subgroup $S(m_1) \times S(m_2) \times \cdots \times S(m_r)$.

 By a slight abuse of notation, when $W_J=H'$ and $W_K=H$ the sets $D_{H',H},D_J,\Delta_J$ will be denoted $D_{H',H},D_{H'},\Delta_{H'}$. The length function $l(\sigma)$ on $S(N)$ is the number of inversions of a permutation, that is, the number of pairs $(i,j)$ such that $i<j$ and $\sigma(i) > \sigma(j)$.

{\color{black}
\textbf{Example 1}. Let $H=S(1) \times S(2) \times S(2) \times S(3)$ and $H'=S(1)\times S(2) \times S(2) \times S(2) \times S(1)$. The set of simple roots are $\Delta_H=\{s_2,s_4,s_6,s_7\}$ and $\Delta_{H'}=\{s_2,s_4,s_6\}$. Set $\sigma = s_5s_4s_3s_1s_6s_5$.  Then
\begin{align*}
\tau_{s_5s_4s_3s_1s_6s_5} s_2 &= s_1 + s_2 + s_3 + s_4 + s_5\\
\tau_{s_5s_4s_3s_1s_6s_5} s_4 &= s_3+s_4+s_5+s_6\\
\tau_{s_5s_4s_3s_1s_6s_5} s_6 &= s_4\\
\tau_{s_5s_4s_3s_1s_6s_5} s_7 &= s_5+s_6+s_7\\
(\tau_{s_5s_4s_3s_1s_6s_5})^{-1} s_4 &= s_6 \\
(\tau_{s_5s_4s_3s_1s_6s_5})^{-1} s_2 &= s_1+s_2+s_3 \\
(\tau_{s_5s_4s_3s_1s_6s_5})^{-1} s_6 &= s_3 + s_4 
\end{align*}
showing that $\tau_{\sigma}(\Delta_H) \subseteq \Phi^+$ and $\tau_{\sigma}^{-1}(\Delta_{H'}) \subseteq \Phi^+$, and thus $\sigma \in D_{H',H}$.  The subgroup $L$ is defined by $\Delta_L=\Delta_{H'} \cap \tau_{\sigma}(\Delta_H)$, and from the above calculations we see that $L$ is the Young subgroup generated by the single element $s_4$, and therefore by Proposition \ref{2.3.3}(b), $H' \cap D_L = \{s_2,s_6\}$. In the next section, we will see how $\sigma$ can be constructed by using the state space of multi--species $q$--TAZRP.
}

\subsection{$q$--notation}

Fix $0<q<1$. For any $k\geq 0$, let 
$$
[k]_q = \frac{1-q^k}{1-q} = 1 + q + q^2 + \ldots + q^{k-1}
$$ 
be the $q$--deformed integer. Let $[k]_q^! = [1]_q \cdots [k]_q$ be the $q$--deformed factorial. The $q$--Pochhammer symbol is
$$
(\alpha;q)_k = (1-\alpha)(1-q\alpha) \cdots ( 1-q^{k-1}\alpha), \quad 0 \leq k \leq \infty.
$$
Observe that $(1-q)^k[k]_q^! = (q;q)_k$.

For each integer $r\geq 1$ and each finite sequence of non--negative integers $\mathbf{m}=(m_1,m_2,\ldots,m_r)$ whose sum is $N=m_1 + \ldots + m_r$, define the $q$--multinomial 
$$
\left[ 
\begin{array}{c}
N\\
\mathbf{m}
\end{array}
\right]_q
:=
\left[ 
\begin{array}{c}
m_1 + \ldots + m_r \\
m_1,m_2,\ldots,m_r
\end{array}
\right]_q
= \frac{[N]^!_q}{ [m_1]^!_q \cdots [m_r]^!_q}.
$$
Given a subgroup $G$ of $S(N)$, let 
$$
\vert G \vert_q = \sum_{\sigma \in G} q^{ l(\sigma) } .
$$

\begin{proposition}\label{QBIN}
Suppose $H$ is the Young subgroup $S(m_1) \times \cdots \times S(m_r)\subseteq S(N)$ where $m_1 + \cdots + m_r = N$. Then
$$
\sum_{\sigma^0 \in D_H} q^{\inv(\sigma_0)} = \left[ 
\begin{array}{c}
m_1 + \ldots + m_r \\
m_1,m_2,\ldots,m_r
\end{array}
\right]_q 
= \frac{ \vert S(N)\vert_q}{ \vert H\vert_q}
$$
If additionally $H'$ is also a Young subgroup and $\sigma$ is some fixed element of $D_{H',H}$, then 
$$
\sum_{a \in H' \cap D_L} q^{\inv(a)} = 
\frac{ \vert H'\vert_q}{\vert L\vert_q}
$$
where $L$ is the Young subgroup generated by $\Delta_L = \Delta_{H'} \cap \sigma(\Delta_H)$.
\end{proposition}
\begin{proof}
The first statement is equivalent to the $q$--Binomial theorem (see e.g. Theorem 3.6 of \cite{INS}). It also follows from the Proposition \ref{2.3.3}. Namely, 
$$
[N]_q^! = \sum_{\sigma \in S(N)} q^{\inv(\sigma)} = \sum_{b \in H} \sum_{\sigma^0 \in D_H} q^{\inv(b) + \inv(\sigma^0)} = [m_1]_q^! \cdots [m_r]_q^! \sum_{\sigma^0 \in D_H} q^{\inv(\sigma^0)}.
$$ 
For the second statement, consider the decomposition $H' = (H' \cap D_L)L$. Then arguing similarly,
$$
\vert H'\vert_q = \sum_{\sigma \in S(H')} q^{\inv(\sigma)} =  \sum_{a \in H' \cap D_L} \sum_{x' \in L} q^{\inv(a) + \inv(x')} = \vert L\vert_q \sum_{a \in H' \cap D_L} q^{\inv(a)}. 
$$

\end{proof}

\subsection{Interpretation as multi--species $q$--TAZRP state space}
The state space for $n$--species $q$--TAZRP consists of particle configurations on a one--dimensional lattice. Here, we take that lattice to be $\mathbb{Z}$. At each lattice site, there may be arbitrarily many particles, with $n$ different species of particles. The state space is therefore $\left(\mathbb{Z}_{\geq 0}^n \right)^{\mathbb{Z}}$. Each $\eta \in \left(\mathbb{Z}_{\geq 0}^n \right)^{\mathbb{Z}}$ can be written as $\eta = (\eta_i^x)$, for $x \in \mathbb{Z}$ and $1 \leq i \leq n$, where $\eta_i^x$ denotes the number of particles of species $i$ located at lattice site $x$. 

In general, a particle configuration can have infinitely many particles. When restricting states with finitely many particles, there is another convenient way of writing particle configurations.  Assume there are $N_k$ particles of species $k$ ($1 \leq k \leq n$). Set $N=N_1 + \ldots + N_n$ to be the total number of particles. Let $\mathbf{N}$ denote $(N_1,\ldots,N_n)$. For $i\leq j$ let $N_{[1,j]}$ denote $N_i + \ldots + N_j$. A particle configuration can be expressed as a pair $(\mathbf{x},\sigma)$, where 
$$
\mathbf{x} = (x_1 \geq x_2 \geq \ldots \geq x_N)
$$
indicates the location of the particles. Let $\sigma \in S(N)$ denote the ordering of the species, in the sense that if $\sigma$ is written in two--line notation as
$$
\left(
\begin{array}{cccc}
\sigma_1 & \sigma_2 & \cdots & \sigma_N\\
1 & 2 & \cdots & N
\end{array}
\right),
$$
so that $\sigma_j = \sigma^{-1}(j)$, then the $N_k$ particles of species $k$ are located at the lattice sites
\begin{equation}\label{Equiv1}
x_{\sigma_{N_{[1,k-1]}+1}}, \ldots, x_{ \sigma_{  N_{[1,k]} }}.
\end{equation}

An equivalent description of the particle configuration $(\mathbf{x},\sigma)$ is as follows. For
$$
\mathbf{x} \in \mathcal{W}_N := \{ (x_1 , \ldots, x_N) : x_1 \geq \ldots \geq x_N  \} \subset \mathbb{Z}^N,
$$
define $\mathbf{m}(\mbfx)=(m_1,\ldots,m_r)$ so that 
$$
x_1 = \cdots = x_{m_1} > x_{m_1+1} = \cdots = x_{m_1+m_2} > x_{m_1+m_2+1} = \cdots = \cdots > x_{m_1+\ldots + m_{r-1}+1} = \cdots = x_N,
$$ 
 where $m_r$ is defined by $m_1+\ldots+m_r=N$.  Also define $k_1,\ldots,k_N$ by $k_1 = \cdots = k_{N_1}=1, k_{N_1+1} = \cdots = k_{N_1+N_2}=2, \ldots$. Then the particles located at the lattice site $x_{m_1 + \ldots + m_s + 1} = \cdots = x_{m_1 + \ldots + m_{s+1}}$ have species 
\begin{equation}\label{Equiv2}
k_{\sigma(m_1 + \ldots + m_s + 1)}, \ldots, k_{\sigma(m_1 + \ldots + m_{s+1})}.
\end{equation}
Note that $\sigma(j)$ is not the same as $\sigma_j$ in this notation.

Because more than one particle can occupy a site, the map $\mathcal{W}_N \times S(N) \rightarrow  (\mathbb{Z}^n_{\geq 0})^{\mathbb{Z}}$ is not injective. 

\textbf{Example 2.} Consider the particle configuration shown in the left side of Figure \ref{State}. There is more than one $\sigma\in S(N)$ which defines this particle configuration, and it is not hard to see that $\sigma = 21467358$ has the fewest inversions. In fact, this $\sigma$ is the element $s_5s_4s_3s_1s_6s_5 \in D_{H',H}$ from the previous example, where the $H=S(\mathbf{m}(\mathbf{x}))=S(1)\times S(2) \times S(2) \times S(3)$ and $H'=S(\mathbf{N})=S(1) \times S(2) \times S(2) \times S(2) \times S(1)$.

It is straightforward to see that the permutations $21476358,21567438, 35178426$ also describe the same particle configuration as $\sigma$, but have more inversions. These turn out to be in the double coset $H'\sigma H$ of $\sigma$, because it can be seen through direct calculation that the decompositions from Proposition \ref{Cox}(b) take the form
\begin{align*}
21476358 &= e\cdot s_5s_4s_3s_1s_6s_5 \cdot s_6,\\
21567438 &= s_6 \cdot s_5s_4s_3s_1s_6s_5 \cdot s_4,\\
35178426 &= s_6s_2 \cdot s_5s_4s_3s_1s_6s_5 \cdot s_7s_6s_4s_2.
\end{align*}
Note that $21476358$ is also equal to $s_4 \cdot s_5s_4s_3s_1s_6s_5 \cdot e$, demonstrating that $s_4 \notin D_L$. These calculations can be found at the end of the document.

The above example is true in general.

\begin{proposition}\label{Wawawawa}
(a) Set $H=S(\mathbf{m}(\mathbf{x}))$ and $H' = S(\mathbf{N})$. For a fixed particle configuration with particles located at $\mathbf{x}$, its fiber is $\{\mathbf{x}\} \times H'\sigma H$ for some $\sigma \in  D_{H',H}$.

(b) Suppose that the particle configuration defined by $(\mathbf{x},\sigma)$ for some $\sigma \in D_{H',H}$ has two particles of the same species $k_i = k_{i+1}$ at the same lattice site. Then $s_i \in L$ where $\Delta_L:=\Delta_{H'} \cap \tau_{\sigma}(\Delta_H)$. 

Conversely, suppose that $\sigma \in D_{H',H}$ and $s_i \in L$,  where $\Delta_L:=\Delta_{H'} \cap \tau_{\sigma}(\Delta_H)$. Then the particle configuration defined by $(\mathbf{x},\sigma)$ has two particles of the same species $k_i=k_{i+1}$ at the same lattice site.

(c) Suppose that $s_i$ satisfies the same assumptions as in (b). Suppose there is some $s_j \in H$ such that $\tau_{\sigma}s_j=s_i$. Then $s_i \sigma = \sigma s_j$. 

Conversely, suppose that $s_i \sigma = \sigma s_j$ for some $s_j \in H$. Then $\tau_{\sigma} s_j = s_i$.

\end{proposition}
\begin{proof}
(a)  By \eqref{Equiv2}, $(\mathbf{x},\sigma)$ and $(\mathbf{x},\sigma b)$ define the same particle configuration if and only if $b\in H$. By \eqref{Equiv1}, $(\mathbf{x},\sigma b)$ and $(\mathbf{x},a\sigma b)$ define the same particle configuration if and only if $a\in H'$. This shows part (a).

(b) Suppose that the lattice site is $x_j=x_{j+1}$. Then $s_i \sigma = \sigma s_j$, violating the uniqueness condition in \ref{Cox}(b), which can only hold if $s_i \notin D_L$. Since $s_i\notin D_L$, then by Proposition \ref{2.3.3}(b) it can be written as $s_i=ax$ for some $a\in D_L$ and some non--identity $x\in L$, such that $l(s_i) = l(a) + l(x)$. But $l(s_i)=1$ and $l(x) \geq 1$, which must imply that $a=e$, and thus $s_i \in L$.

Conversely, suppose that $s_i \in L$. Then $s_i \notin D_L$, and so by the uniqueness property of Proposition \ref{Cox}(b), $s_i \sigma = a\sigma b$ where either $a \in H' \cap D_L$ or $b\in H$. The element $b$ cannot be the identity element $e$, for otherwise $s_i = a \in D_L$. Therefore $l(s_i \sigma) = l(a) + l(\sigma) + l(b) \geq l(\sigma) + 1$, so by Proposition \ref{Three}(c) the element $a$ must be the identity and the element $b$ must be some $s_j \in H$. Therefore the lattice site $x_j=x_{j+1}$ contains two particles of species $k_i = k_{i+1}$.

(c) Assume that $\tau_{\sigma}s_j = s_i$. Then $\tau_{s_i \sigma}s_j = -s_i$, so by Proposition \ref{Three}(b), $l(s_i \sigma) = l(\sigma) + 1$. By the unique decomposition, $s_i \sigma = \sigma s_l$ for some $s_l \in H$. But this implies that $\tau_{\sigma}\tau_{s_l} s_j  = \tau_{\sigma s_l} s_j = \tau_{s_i \sigma}s_j = -s_i$, so by assumption $\tau_{s_l}s_j = -s_j$. This can only hold if $s_l=s_j$, as needed.

For the converse implication, we have that $l(s_i \sigma) = 1 + l(\sigma)$. Therefore $\tau_{s_i}$ takes a positive root in $\tau_{\sigma}(\Phi^+)$ and makes it negative. By Proposition \ref{Three}(a), this root must be $s_i$. In other words, $s_i \in \tau_{\sigma}(\Phi^+)$ and $s_i \notin \tau_{s_i \sigma}(\Phi^+) = \tau_{\sigma s_j}(\Phi^+)$. Therefore $\tau_{\sigma}^{-1}(s_i) \in \Phi^+$ and $\tau_{s_j}\tau_{\sigma}^{-1} (s_i )\notin \Phi^+$. This implies that $\tau_{\sigma}^{-1} s_i = s_j$, as needed.
\end{proof}

\textbf{Remark}. Part (a) of the proposition appears to be equivalent to results in \cite{Jo}, where the bijection is stated in terms of arrays rather than particle configurations. Part (c) can also be seen by comparing Theorem 2.7.4 of \cite{BOOK} and Lemma 2 of \cite{Solomon}, which had been previously announced in \cite{Tits1}, p. 26 and appeared (in geometric form) in \cite{Tits2}, section 12.2.

The previous theorem motivates the following definition. Given a particle configuration $(\mathbf{x},\sigma)$, where $\mathbf{m}(\mathbf{x})=(m_1,\ldots,m_r)$ and the species numbers are given by $(N_1,\ldots,N_n)$, let $L_{ij}$ denote the number of species $j$ particles located at the lattice site $x_{m_1 + \ldots + m_i}$. Here the ranges of $i$ and $j$ are given by $1 \leq i \leq r$ and $1 \leq j \leq n$.

\begin{figure}
\begin{center}
\begin{tikzpicture}[scale=0.9, every text node part/.style={align=center}]
\usetikzlibrary{arrows}
\usetikzlibrary{shapes}
\usetikzlibrary{shapes.multipart}

\draw (-1,0)--(3,0);
\draw (-1,0)--(-1,0.66);
\draw (-0.33,0)--(-0.33,0.66);
\draw (0.33,0)--(0.33,0.66);
\draw (1,0)--(1,0.66);
\draw (1.66,0)--(1.66,0.66);
\draw (2.33,0)--(2.33,0.66);
\draw (3,0)--(3,0.66);
\draw (0.66,0.33) circle (8pt);
\draw (0,0.33) circle (8pt);
\draw (0,1) circle (8pt);
\draw (0,1.66) circle (8pt);
\draw (0.66,1) circle (8pt);
\draw (2,1) circle (8pt);
\draw (2,0.33) circle (8pt);
\draw (2.66,0.33) circle (8pt);

\node at (0,1.66) {$5$};
\node at (0,1) {$3$};
\node at (0,0.33) {$3$};
\node at (0.66,0.33) {$2$};
\node at (0.66,1)  {$4$};
\node at (2,1) {$4$};
\node at (2,0.33) {$1$};
\node at (2.66,0.33) {$2$};

\draw  (6,0.5) circle (8pt);
\draw  (6.66,0.5) circle (8pt);
\draw  (7.33,0.5) circle (8pt);
\draw  (8,0.5)  circle (8pt);
\draw  (8.66,0.5) circle (8pt);
\draw  (9.33,0.5) circle (8pt);
\draw  (10,0.5) circle (8pt);
\draw  (10.66,0.5) circle (8pt);

\node at (5.3,1.16) {$\sigma=$};
\node at (6,1.16) {$2$};
\node at (6.66,1.16) {$1$};
\node at (7.33,1.16) {$4$};
\node at (8,1.16)  {$6$};
\node at (8.66,1.16) {$7$};
\node at (9.33,1.16) {$3$};
\node at (10,1.16) {$5$};
\node at (10.66,1.16) {$8$};

\node at (6,0.5) {$1$};
\node at (6.66,0.5) {$2$};
\node at (7.33,0.5) {$2$};
\node at (8,0.5)  {$3$};
\node at (8.66,0.5) {$3$};
\node at (9.33,0.5) {$4$};
\node at (10,0.5) {$4$};
\node at (10.66,0.5) {$5$};

\end{tikzpicture}
\end{center}
\caption{The particle configuration referenced in Example 2.
}
\label{State}
\end{figure}
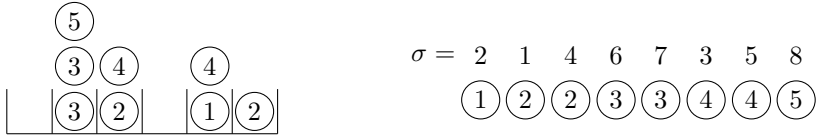

\section{$q$--TAZRP}\label{2}

\subsection{Dynamics}
Let us define the dynamics of the multi--species $q$--TAZRP. A visual example is shown in Figure \ref{TAZER}.

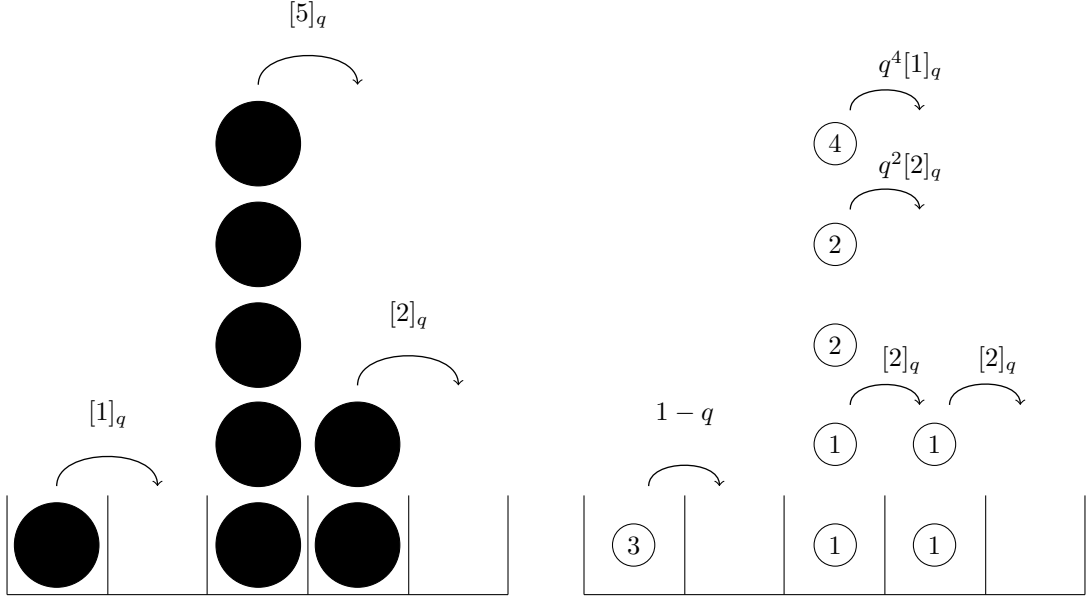
\begin{figure}
\begin{center}
\begin{tikzpicture}[scale=2, every text node part/.style={align=center}]
\usetikzlibrary{arrows}
\usetikzlibrary{shapes}
\usetikzlibrary{shapes.multipart}
\tikzstyle{arrow}=[->,>=stealth,thick,rounded corners=4pt]

\draw (0.33,0)--(3.66,0);
\draw (0.33,0)--(0.33,0.66);
\draw (1,0)--(1,0.66);
\draw (1.66,0)--(1.66,0.66);
\draw (2.33,0)--(2.33,0.66);
\draw (3,0)--(3,0.66);
\draw (3.66,0)--(3.66,0.66);
\draw[fill=black] (0.66,0.33) circle (8pt);
\draw[fill=black] (2,0.33) circle (8pt);
\draw[fill=black] (2,1) circle (8pt);
\draw[fill=black] (2,1.66) circle (8pt);
\draw[fill=black] (2,2.33) circle (8pt);
\draw[fill=black] (2,3) circle (8pt);
\draw[fill=black] (2.66,0.33) circle (8pt);
\draw[fill=black] (2.66,1) circle (8pt);

\node at (2,3.33) (abc) {};
\node at (2.66,3.33) (abd) {};
\node at (2.33,3.86) {$[5]_q$};
\node at (2.66,1.33) (abe) {};
\node at (3.33,1.33) (abf) {};
\node at (3,1.86) {$[2]_q$};
\node at (0.66,0.66) (abx) {};
\node at (1.33,0.66) (aby) {};
\node at (1,1.2) {$[1]_q$};
\draw (abc) edge[out=90,in=90,->, line width=0.5pt] (abd);
\draw (abe) edge[out=90,in=90,->, line width=0.5pt] (abf);
\draw (abx) edge[out=90,in=90,->, line width=0.5pt] (aby);
\end{tikzpicture}
\hspace{0.3in}
\begin{tikzpicture}[scale=2, every text node part/.style={align=center}]
\usetikzlibrary{arrows}
\usetikzlibrary{shapes}
\usetikzlibrary{shapes.multipart}
\tikzstyle{arrow}=[->,>=stealth,thick,rounded corners=4pt]

\draw (0.33,0)--(3.66,0);
\draw (0.33,0)--(0.33,0.66);
\draw (1,0)--(1,0.66);
\draw (1.66,0)--(1.66,0.66);
\draw (2.33,0)--(2.33,0.66);
\draw (3,0)--(3,0.66);
\draw (3.66,0)--(3.66,0.66);
\draw (0.66,0.33) circle (4pt);
\draw  (2,0.33) circle (4pt);
\draw  (2,1) circle (4pt);
\draw  (2,1.66) circle (4pt);
\draw  (2,2.33) circle (4pt);
\draw (2,3) circle (4pt);
\draw  (2.66,0.33) circle (4pt);
\draw  (2.66,1) circle (4pt);

\node at (0.66,0.33) {$3$};
\node  at (2,0.33) {$1$};
\node  at (2,1) {$1$};
\node  at (2,1.66) {$2$};
\node  at (2,2.33) {$2$};
\node at (2,3) {$4$};
\node at (2.66,0.33) {$1$};
\node at (2.66,1) {$1$};

\node at (2.1,3.16) (abc) {};
\node at (2.56,3.16) (abd) {};
\node at (2.50,3.52) {$q^4[1]_q$};
\node at (2.1,2.5) (abM) {};
\node at (2.56,2.5) (abN) {};
\node at (2.50,2.85) {$q^2[2]_q$};
\node at (2.1,1.2) (wtf) {};
\node at (2.56,1.2) (lol) {};
\node at (2.76,1.2) (abe) {};
\node at (3.23,1.2) (abf) {};
\node at (3.08,1.56) {$[2]_q$};
\node at (2.44,1.56) {$[2]_q$};
\node at (0.76,0.66) (abx) {};
\node at (1.23,0.66) (aby) {};
\node at (1,1.2) {$1-q$};
\draw (abc) edge[out=90,in=90,->, line width=0.5pt] (abd);
\draw (abM) edge[out=90,in=90,->, line width=0.5pt] (abN);
\draw (abe) edge[out=90,in=90,->, line width=0.5pt] (abf);
\draw (abx) edge[out=90,in=90,->, line width=0.5pt] (aby);
\draw (wtf) edge[out=90,in=90,->, line width=0.5pt] (lol);
\end{tikzpicture}

\end{center}

\caption{Example of jump rates for the homogeneous $q$--TAZRP, with time scaled by $1-q$. The left image shows single--species and the right image shows multi--species. Note that $q^4[1]_q + q^2[2]_q + [2]_q  = [5]_q$.}
\label{TAZER}
\end{figure}

For a particle configuration $\xi = (\xi_i^x)$, the jump rates for an $i$th species particle at lattice site $x$ to jump one step to the right is 
$$
b_x \cdot q^{\xi_1^x + \ldots + \xi_{i-1}^x} [\xi_i^x]_q,
$$
where $(b_x)_{x\in \mathbb{Z}}$ are the inhomogeneity parameters, which we assume to be positive uniformly bounded.

Let us define the generator explicitly. if $\eta=(\eta_i^x)$ and $\xi=(\xi_i^x)$ are related by 
$$
\eta_j^y = \xi_j^y -1, \quad \eta_j^{y+1} = \xi_j^{y+1}+1
$$
for some $j,y$, with all $\eta_i^x = \xi_i^x$ for all other values of $i,x$, then write $\eta=\xi(j,y)$. Then the generator is defined by 
$$
L_{\text{mqTAZRP}}(\xi, \xi(j,y)) = b_y q^{\xi_1^y + \ldots + \xi_{j-1}^y} [\xi_j^y]_q .
$$
If the particle configuration $\xi(j,y)$ does not exist, then formally set $\xi(j,y)=\xi$.

\subsection{Initial conditions}
We now define a class of probability measures on the multi--species $q$--TAZRP state space. By Proposition \ref{Wawawawa}, a particle configuration with $N$ particles can be written uniquely as $(\mathbf{x},\sigma)$ where $\sigma \in D_{H',H}$, $H'=S(\mathbf{N}),H=S(\mathbf{m}(\mathbf{x}))$. 

\begin{definition}
A probability measure on $\displaystyle\coprod_{\mathbf{x} \in \mathcal{W}_N} \{\mathbf{x}\} \times D_{H',H}$ is $q$--exchangeable if 
$$
q^{-l(\sigma)} \cdot  \text{Prob}(\mathbf{x},\sigma)  = q^{-l(\sigma')}  \cdot \text{Prob}(\mathbf{x},\sigma') 
$$
for all $\mathbf{x}\in \mathcal{W}_N$ and $\sigma,\sigma' \in D_{H',H}$. 

If a $q$--exchangeable probability measure satisfies
$$
\text{Prob}(\mathbf{x},\sigma) = 1_{\{\mathbf{x} = \mathbf{y}\}} \frac{q^{l(\sigma)}}{Z},
$$
then we say that it is supported on $\mathbf{y}$.
\end{definition}

See \cite{GO,GO2} for a general treatment of $q$--exchangeable measures.

\subsection{Markov Projections}
From the identity
\begin{equation}\label{id}
[a]_q + q^a[b]_q = [a+b]_q,
\end{equation}
it is not hard\footnote{ It was first explicitly stated and proved in \cite{K} for the homogeneous case (see also \cite{K17} for a more general totally asymmetric zero range process), but the proof for the in homogeneous multi--species $q$--TAZRP is not hard.} to see that there is a Markov projection property.  Namely, suppose $\Pi$ is a partition of $\{1,\ldots,n\}$ into $k$ blocks of consecutive integers:
$$
\Pi = \{ \Pi_1, \ldots, \Pi_k\} = \{ \{ 1, \ldots, p_1\}, \{p_1 + 1, \ldots, p_1 + p_2\}, \ldots \{p_1 + \ldots + p_{k-1} + 1, \ldots, n\} \}.
$$
Then there is a corresponding projection 
$$
\pi: (\mathbb{Z}_{\geq 0}^n )^{\mathbb{Z}} \rightarrow (\mathbb{Z}_{\geq 0}^k )^{\mathbb{Z}}
$$ 
such that for $\eta = \pi(\xi)$,
$$
\eta_i^x = \sum_{j \in \Pi_i} \xi_j^x, \quad \text{ for } 1 \leq i \leq k.
$$
The Markov projection property here says that
$$
L_{\text{mqTAZRP}}( \pi(\xi) ,\eta  ) = \sum_{ \psi \in \pi^{-1}(\eta)} L_{\text{mqTAZRP}}(  \xi , \psi ), \quad \quad \eta \in (\mathbb{Z}_{\geq 0}^k)^{\mathbb{Z}}, \quad \xi \in (\mathbb{Z}_{\geq 0}^n)^{\mathbb{Z}}.
$$

\subsection{Previous results from single--species $q$--TAZRP}
\subsubsection{Stationary measures}
For a (single--species) zero range process with certain mild conditions, the stationary measures can be found (see \cite{AND}). When $n=1$, and all $b_y$ are equal (i.e. in the homogeneous case), the stationary measures are given by
\begin{equation}\label{Stationary}
\mathbb{P}^{\alpha}(\xi^x_1 = k) = (\alpha;q)_{\infty} \frac{\alpha^k}{(q;q)_k},
\end{equation}
where $\alpha \in [0,1)$ is a parameter. Note that when $\alpha\rightarrow 1$, the normalization factor $(\alpha;q)_{\infty}$ equals $0$. Thus, even though $\mathbb{P}^1$ still defines a measure preserved by the generator, but it is not a probability measure. 

%In other words, if $p(\xi)$ is the function given by given by 
%$$
%p(\xi) = \prod_{x \in \mathbb{Z}} \frac{1}{(q;q)_{\xi^x_1}}, 
%$$
%then
%$$  -\mathrm{wt}(\mathbf{x}) p( (\mathbf{x},w))  +  \sum_{\{s\}} L( (\mathbf{x}, w)^{-s}, (\mathbf{x},w)) p( (\mathbf{x},w)^{-s})
%=0.
%$$
%This can be seen directly from 
%$$
%-\frac{\sum_{k=1}^r [m_k]_q}{[m_1]_q^! [m_2]_q^! \cdots [m_r]_q^!} + \sum_{k=1}^r \frac{1}{{[m_1]_q^! \cdots [m_{k-1}]_q^! [m_k-1]_q^! [m_{k+1}]_q^! \cdots [m_r]_q^!}}=0.
%$$
%To see that this is the right equation, note that for $x_{m_k}-1 = x_{m_{k+1}}$, the $[m_{k-1}]_q^!$ in the denominator should be replaced with $[m_{k-1}+1]_q^!$ (from the formula for $\mathbb{P}^{\alpha}$), but this cancels $[m_{k-1}]_q$ occuring in the numerator (from the jump rates).

\subsubsection{Transition Probabilities for the single--species $q$--TAZRP}
Here, let 
$$
S_{(\beta,\alpha)} = - \frac{q w_{\beta} - w_{\alpha}}{q w_{\alpha} - w_{\beta}}
$$
and set
$$
A_{\sigma} = \prod_{(\beta,\alpha) \text{ is an inversion of } \sigma} S_{(\beta,\alpha)}.
$$
By Theorem 1.1 of \cite{WW}, which generalizes the homogeneous case proved in Theorem 2.6 of \cite{KL}, given an initial condition $Y=(y_1,\ldots,y_N)$ and another particle configuration $X=(x_1,\ldots,x_N)$, the transition probabilities are 
\begin{equation}\label{Wawa}
P_Y(X;t) =  \frac{1}{[N]_q^!}  \left[ 
\begin{array}{c}
N\\
\mathbf{m}(\mbfx)
\end{array}
\right]_q \left( \frac{1}{2\pi i}\right)^N \int_{\mathcal{C}_R} \cdots \int_{\mathcal{C}_R} \sum_{\sigma \in S_N} A_{\sigma} \prod_{j=1}^N \left[ \prod_{k=y_{\sigma(j)}}^{x_j} \left( \frac{b_k}{b_k - w_{\sigma(j)}}\right) e^{-w_jt}  \right] dw_1 \cdots dw_N.
\end{equation}
where the contours are counterclockwise circles centered at the origin with large radius $R$. The poles $b_k$ and $qw_i$ need to be enclosed, but not $q^{-1}w_l$, where $i<j$ and $l>j$. Recall that the $b_k$ need to be bounded. The usual definition of the product is extended to 
$$
\prod_{k=m}^n f(k) 
= 
\begin{cases}
\prod_{k=m}^n f(k), & \text{ if } n \geq m \\
1, & \text{ if } n=m-1 \\
\prod_{k=n+1}^{m-1} \frac{1}{f(k)}, & \text{ if } n<m-1.
\end{cases}
$$

In Remark 1.2 of \cite{WW}, it is remarked that the pre--factor $\dfrac{1}{[N]_q^!}  \left[ 
\begin{array}{c}
N\\
\mathbf{m}(\mbfx)
\end{array}
\right]_q$ 
is related to the stationary measures of the $q$--TAZRP, with a brief explanation given. A similar statement will hold for the multi--species $q$--TAZRP. 

\subsection{Statement of Main Theorems}
Recall the definition of $L_{ij}$ at the end of section \ref{1}.
\begin{theorem}\label{Main}
Let $\mathbf{N}=(N_1,\ldots,N_n)$. Fix $\mathbf{y}=(y_1,\ldots,y_N)$ and $\mathbf{x}=(x_1,\ldots,x_N)$, and let $\sigma \in D_{H',H}$ where $H'=S(\mathbf{N})$ and $H=S(\mathbf{m}(\mathbf{x}))$. Given $q$--exchangeable initial conditions supported at $\mathbf{y}$,
\begin{multline*}
\mathrm{Prob}((\mathbf{x},\sigma) \text{ at time } t) = \frac{q^{\inv(\sigma)}}{[N]_q^!} \cdot  \left( \frac{ \prod_{j=1}^n [N_j]_q^!}{ \prod_{j=1}^n \prod_{i=1}^r [L_{ij}]_q^!}\right) \\
\times  \left( \frac{1}{2\pi i}\right)^N \int_{\mathcal{C}_R} \cdots \int_{\mathcal{C}_R} \sum_{\tau \in S_N} A_{\tau} \prod_{j=1}^N \left[ \prod_{k=y_{\tau(j)}}^{x_j} \left( \frac{b_k}{b_k - w_{\tau(j)}}\right) e^{-w_jt}  \right] dw_1 \cdots dw_N
\end{multline*}
\end{theorem}

\begin{remark}

When $\mathbf{N}=(N)$, the product in the denominator simplifies to $\prod_{i=1}^r [m_i]_q^!$, and $D_{H',H}=\{e\}$, so the theorem reduces to the single--species case.  If $\mathbf{N}=(1,\ldots,1)$, then the $q$--multinomial terms are all equal to $1$. 
\end{remark}

\begin{theorem}\label{lamelol}
For $\alpha\in [0,1)$, let $\mathbb{P}^{(\alpha)}$ be the stationary measure for homogeneous single species $q$--TAZRP, defined by \eqref{Stationary}. Then the $q$--exchangeable probability measure defined by 
$$
\mathrm{Prob}(\mathbf{x},\sigma) =\frac{q^{l(\sigma)}}{Z} \  \mathbb{P}^{(\alpha)}(\mathbf{x}) 
$$
is a stationary measure for the homogeneous multi--species $q$--TAZRP.
\end{theorem}

\subsection{Proof of main theorems}

\subsubsection{The simplest case}
Consider the simplest case, when there is exactly one particle of each species. From a probabilistic perspective, one might expect the ``simplest'' case to be when there is only one particle of species number $2$ with other particles of species number $1$. However, from the algebraic perspective, the subgroup $H'=S(\mathbf{N})\subseteq S(N)$ is trivial when $H'=\{e\}$ or $S(N)$, and the former situation corresponds to having exactly one particle of each species, while the latter situation corresponds to having only species $1$ particles. 

By the Markov projection property, the more general case follows from this simple case immediately. To see this, suppose that we have already shown Theorem \ref{Main} when $\mathbf{N}=(1,\ldots,1)$.  Then $\sigma^0$ can take any value in $D_H$, and for any $\sigma^0\in D_H$,
\begin{multline*}
\mathrm{Prob}((\mathbf{x},\sigma^0) \text{ at time } t) \\= \frac{q^{\inv(\sigma^0)}}{[N]_q^!}  \left( \frac{1}{2\pi i}\right)^N \int_{\mathcal{C}_R} \cdots \int_{\mathcal{C}_R} \sum_{\tau \in S_N} A_{\tau} \prod_{j=1}^N \left[ \prod_{k=y_{\tau(j)}}^{x_j} \left( \frac{b_k}{b_k - w_{\tau(j)}}\right) e^{-w_jt}  \right] dw_1 \cdots dw_N
\end{multline*}

Now consider the general case of $\mathbf{N}=(N_1,\ldots,N_n)$ and set $H'=S(\mathbf{N})$. The projection map $(\mathbb{Z}_{\geq 0}^N )^{\mathbb{Z}}\rightarrow (\mathbb{Z}_{\geq 0}^N )^{\mathbb{Z}}$ can be expressed in terms of permutations in $S(N)$ in the following way. By Lemma \ref{Moo},  any $\sigma^0 \in D_H$ can be uniquely decomposed as $\sigma^0=a \sigma$ where $a \in H'\cap D_L$ and $\sigma \in D_{H',H}$, and in this decomposition
$$
\pi(\mathbf{x},\sigma^0) = (\mathbf{x},\sigma)
$$ 
Lemma \ref{Moo} also implies that for any fixed $\sigma \in D_{H',H}$ the set of all $\sigma^0$ such that $\pi(\mathbf{x},\sigma^0) = (\mathbf{x},\sigma)$ is precisely the set of $\sigma^0 \in D_H$ which can be decomposed as $\sigma^0 = a\sigma$ for $a\in H' \cap D_L$. Thus, the Markov projection property then implies that
\begin{align*}
\mathrm{Prob}&((\mathbf{x},\sigma) \text{ at time } t)  = \sum_{a \in H' \cap D_L} \mathrm{Prob}((\mathbf{x},a\sigma) \text{ at time } t)  \\ 
&= \sum_{a \in H' \cap D_L}\frac{q^{\inv(a\sigma)}}{[N]_q^!}  \left( \frac{1}{2\pi i}\right)^N \int_{\mathcal{C}_R} \cdots \int_{\mathcal{C}_R} \sum_{\tau \in S_N} A_{\tau} \prod_{j=1}^N \left[ \prod_{k=y_{\tau(j)}}^{x_j} \left( \frac{b_k}{b_k - w_{\tau(j)}}\right) e^{-w_jt}  \right] dw_1 \cdots dw_N\\
&= q^{\inv(\sigma)}\sum_{a \in H' \cap D_L} \frac{q^{\inv(a)}}{[N]_q^!}  \left( \frac{1}{2\pi i}\right)^N \int_{\mathcal{C}_R} \cdots \int_{\mathcal{C}_R} \sum_{\tau \in S_N} A_{\tau} \prod_{j=1}^N \left[ \prod_{k=y_{\tau(j)}}^{x_j} \left( \frac{b_k}{b_k - w_{\tau(j)}}\right) e^{-w_jt}  \right] dw_1 \cdots dw_N
\end{align*}
Applying Proposition \ref{QBIN} to the sum over $H' \cap D_L$ shows the more general case. 

It thus remains only to prove Theorem \ref{Main} when $\mathbf{N}=(1,\ldots,1)$.
\subsubsection{Master Equation when $\mathbf{N}=(1,\ldots,1)$}

The dynamics can also be defined by stating the master equation. Recall from the definition of the generator, that for a particle configuration $\xi$, the particle configuration obtained by moving a species $j$ particle from $y$ to $y+1$ is denoted $\xi(j,y)$. Let $\xi^{-(j,y)}$ denote the particle configuration obtained by moving a species $j$ particle from $y$ to $y-1$. The master equation is then (where the ``at time'' is omitted)
$$
\frac{d}{dt} \mathrm{Prob}( (\mathbf{x},\sigma);t) =  -\mathrm{wt}(\mathbf{x}) \mathrm{Prob}( (\mathbf{x},\sigma);t)  +  \sum L_{mq\text{TAZRP}}( (\mbfx,\sigma)^{-(j,y)}, (\mathbf{x},\sigma)) \cdot \mathrm{Prob}( (\mathbf{x},\sigma)^{-(j,y)};t) ,
$$
where the sum is taken over all $(j,y)$ such that the corresponding particle configuration is still well--defined. The quantity $\mathrm{wt}(\mathbf{x})$ is the inverse of the expected amount of time the particle configuration spends at $(\mathbf{x},\sigma)$, and is equal to 
\begin{equation}\label{WaitTime}
\mathrm{wt}(\mathbf{x}) = \sum_{s=1}^r b_{x_{m_1 + \ldots + m_s}}[m_s]_q .
\end{equation}

\begin{figure}
\begin{center}
\begin{tikzpicture}[scale=0.9, every text node part/.style={align=center}]
\usetikzlibrary{arrows}
\usetikzlibrary{shapes}
\usetikzlibrary{shapes.multipart}
\tikzstyle{arrow}=[->,>=stealth,thick,rounded corners=4pt]

\draw (-0.33,0)--(3,0);
\draw (-0.33,0)--(-0.33,0.66);
\draw (0.33,0)--(0.33,0.66);
\draw (1,0)--(1,0.66);
\draw (1.66,0)--(1.66,0.66);
\draw (2.33,0)--(2.33,0.66);
\draw (3,0)--(3,0.66);
\draw (0.66,0.33) circle (8pt);
\draw (0,0.33) circle (8pt);
\draw (0.66,1) circle (8pt);
\draw (0.66,1.66) circle (8pt);
\draw (2,1) circle (8pt);
\draw (2,0.33) circle (8pt);
\draw (2.66,0.33) circle (8pt);

\node at (0,0.33) {$2$};
\node at (0.66,0.33) {$1$};
\node at (0.66,1)  {$6$};
\node at (0.66,1.66)  {$7$};
\node at (2,1) {$5$};
\node at (2,0.33) {$3$};
\node at (2.66,0.33) {$4$};

\draw[arrow] (3.33,0.33)--(5,0.33) node[midway,above,fill=white]{$i=4$};

\draw[arrow] (3.33,0)--(5,-1.66) node[midway,above,fill=white]{$i=5$};

\draw (5.66,0)--(9,0);
\draw (5.66,0)--(5.66,0.66);
\draw (6.33,0)--(6.33,0.66);
\draw (7,0)--(7,0.66);
\draw (7.66,0)--(7.66,0.66);
\draw (8.33,0)--(8.33,0.66);
\draw (9,0)--(9,0.66);
\draw (6.66,0.33) circle (8pt);
\draw (6,0.33) circle (8pt);
\draw (6.66,1) circle (8pt);
\draw (6,1) circle (8pt);
\draw (8,1) circle (8pt);
\draw (8,0.33) circle (8pt);
\draw (8.66,0.33) circle (8pt);

\node at (6,0.33) {$1$};
\node at (6.66,0.33) {$6$};
\node at (6.66,1)  {$7$};
\node at (6,1)  {$2$};
\node at (8,1) {$5$};
\node at (8,0.33) {$3$};
\node at (8.66,0.33) {$4$};

\draw (5.66,-2)--(9,-2);
\draw (5.66,-2)--(5.66,-1.33);
\draw (6.33,-2)--(6.33,-1.33);
\draw (7,-2)--(7,-1.33);
\draw (7.66,-2)--(7.66,-1.33);
\draw (8.33,-2)--(8.33,-1.33);
\draw (9,-2)--(9,-1.33);
\draw (6.66,-1.66) circle (8pt);
\draw (6,-1.66) circle (8pt);
\draw (6.66,-1) circle (8pt);
\draw (6,-1) circle (8pt);
\draw (8,-1) circle (8pt);
\draw (8,-1.66) circle (8pt);
\draw (8.66,-1.66) circle (8pt);

\node at (6,-1.66) {$2$};
\node at (6.66,-1.66) {$1$};
\node at (6.66,-1)  {$7$};
\node at (6,-1)  {$6$};
\node at (8,-1) {$5$};
\node at (8,-1.66) {$3$};
\node at (8.66,-1.66) {$4$};

\end{tikzpicture}

\end{center}
\caption{The particule configuration on the left corresponds to $(\mathbf{x},\sigma)$ where $ \mathbf{x}=(5,4,4,2,2,2,1)$ and $\sigma=4351672$.  }
\label{Ex2}
\end{figure}
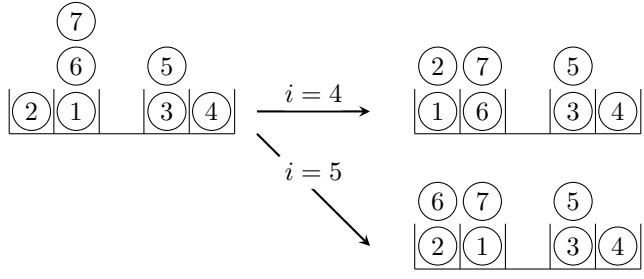

When $\mathbf{N}=(1,\ldots,1)$, the master equation and generator have a nice expression in terms of permutations $\sigma \in S(N)$. Given $\mathbf{x} \in \mathcal{W}_N$ and $1 \leq k \leq N$, the sequence $(x_1,\ldots,x_{k-1},x_k-1,x_{k+1},\ldots,x_N)$ is generally not in $\mathcal{W}_N$. However, rearranging the terms in the sequence produces a unique element of $\mathcal{W}_N$, which we denote $\mathbf{x}^{-k}$. Consider
$$
\overline{\sigma} = \sigma  s_{k+d} s_{k+d-1}\cdots s_k,
$$
where $d$ is the largest integer such that $s_k,\ldots,s_{k+d}$ are all in $H$. Let $\hat{H} = S(\mathbf{m}(\mathbf{x}^{-k}))$. In general, $\overline{\sigma}$ need not be an element of $D_{H',\hat{H}}=D_{\hat{H}}$. Let $\hat{\sigma}$ be the unique element of $D_{\hat{H}}$ such that its coset $H' \hat{\sigma}\hat{H} = \hat{\sigma}\hat{H}$ contains $\overline{\sigma}$. Then $\overline{\sigma}= \hat{\sigma}b$ for some $b \in \hat{H}$, and $l(\overline{\sigma}) = l(\hat{\sigma}) + l(b)$. The generator has the expression
$$
L_{mq\text{TAZRP}}((\mathbf{x}^{-k},\hat{\sigma}), (\mathbf{x},\sigma)) = q^{l(b)} = q^{l(\overline{\sigma})-l(\hat{\sigma})} %=q^{l(\hat{\sigma}^{-1}\overline{\sigma}) }.
$$
The particle configuration $(\mathbf{x}^{-k},\hat{\sigma})$ should be understood as being obtained from $\mathbf{x}^{},\hat{\sigma})$ by moving the $k$th particle one step to the left. See Figure \ref{Ex} for an example.

\begin{lemma}
Let $\tau = \sigma s_j$ and suppose that $l(\tau) = l(\sigma)-1$. Then
$$
q^{l(\hat{\tau})}L_{mq\text{TAZRP}}((\mathbf{x}^{-k},\hat{\tau}), (\mathbf{x},\tau)) = q^{-1} q^{l(\hat{\sigma})}L_{mq\text{TAZRP}}((\mathbf{x}^{-k},\hat{\sigma}), (\mathbf{x},\sigma)) 
$$
\end{lemma}
\begin{proof}
It suffices to show that
$$
l(\overline{\tau}) = l(\overline{\sigma}) - 1.
$$
By definition, $\overline{\tau}= \sigma s_{j} s_{k+d} \cdots s_k$. By Lemma \ref{PREV}, both $\sigma$ and $\sigma s_j$ are in $D_{H',H}=D_H$, so 
$$
l(\overline{\tau}) = l(\sigma s_j) + d+1, \quad l(\overline{\sigma}) = l(\sigma) + d + 1.
$$
Since $l(\sigma _j) = l(\sigma)-1$ by assumption, this proves the lemma.
\end{proof}

We now show that
$$
 q^{ - \inv(\sigma)}   \cdot \mathrm{Prob}((\mathbf{x},\sigma);t) =  q^{-\inv(\tau)}  \cdot \mathrm{Prob}((\mathbf{x},\tau) ;t).
$$
Proceed by induction on the value of $x_1+ \ldots + x_N$. The base case is when $\mathbf{x}=\mathbf{y}$, in which case the result follows because the initial condition is $q$--exchangeable.

\begin{figure}
\begin{center}
\begin{tikzpicture}[scale=0.9, every text node part/.style={align=center}]
\usetikzlibrary{arrows}
\usetikzlibrary{shapes}
\usetikzlibrary{shapes.multipart}

\draw (-0.33,0)--(3,0);
\draw (-0.33,0)--(-0.33,0.66);
\draw (0.33,0)--(0.33,0.66);
\draw (1,0)--(1,0.66);
\draw (1.66,0)--(1.66,0.66);
\draw (2.33,0)--(2.33,0.66);
\draw (3,0)--(3,0.66);
\draw (0.66,0.33) circle (8pt);
\draw (0,0.33) circle (8pt);
\draw (0,1) circle (8pt);
\draw (0.66,1) circle (8pt);
\draw (2,1) circle (8pt);
\draw (2,0.33) circle (8pt);
\draw (2.66,0.33) circle (8pt);

\node at (0,0.33) {$3$};
\node at (0,1) {$7$};
\node at (0.66,0.33) {$4$};
\node at (0.66,1)  {$6$};
\node at (2,1) {$5$};
\node at (2,0.33) {$1$};
\node at (2.66,0.33) {$2$};

\draw (5.66,0)--(9,0);
\draw (5.66,0)--(5.66,0.66);
\draw (6.33,0)--(6.33,0.66);
\draw (7,0)--(7,0.66);
\draw (7.66,0)--(7.66,0.66);
\draw (8.33,0)--(8.33,0.66);
\draw (9,0)--(9,0.66);
\draw (6.66,0.33) circle (8pt);
\draw (6,0.33) circle (8pt);
\draw (6,1) circle (8pt);
\draw (6,1.66) circle (8pt);
\draw (8,1) circle (8pt);
\draw (8,0.33) circle (8pt);
\draw (8.66,0.33) circle (8pt);

\node at (6,0.33) {$3$};
\node at (6,1) {$4$};
\node at (6,1.66) {$7$};
\node at (6.66,0.33) {$6$};
\node at (8,1) {$5$};
\node at (8,0.33) {$1$};
\node at (8.66,0.33) {$2$};

\end{tikzpicture}
\end{center}
\caption{The particle configuration on the left corresponds to $\sigma=2164357$. Set $k=4$. Then $\overline{\sigma} := \sigma \cdot s_4 = 2165347$, which describes the same particle configuration on the left. (One should imagine that the {\textcircled{\small 4}} and the {\textcircled{\small 6}} have switched places).  The particle configuration on the right corresponds to $\hat{\sigma} = 2156347$, and note that $\overline{\sigma} = \hat{\sigma} \cdot s_5$. The jump rate can be determined entirely from $\overline{\sigma}$ and $\hat{\sigma}$.}
\label{Ex}
\end{figure}
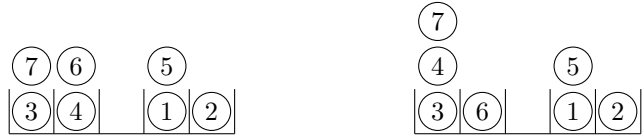

It suffices to consider the case when $\tau=\sigma s_j$ and $l(\tau)=l(\sigma)-1$. In order to apply the induction hypothesis, we use the lemma to rewrite the master equation as 
\begin{align*}
\frac{d}{dt} \mathrm{Prob}( (\mathbf{x},\sigma);t) &=  -\mathrm{wt}(\mathbf{x}) \mathrm{Prob}( (\mathbf{x},\sigma);t)  +  \sum_{k=1}^N L_{mq\text{TAZRP}}(  (\mathbf{x}^{-k},\hat{\sigma}),(\mbfx,\sigma)) \cdot \mathrm{Prob}( (\mathbf{x}^{-k},\hat{\sigma});t) \\
&=-\mathrm{wt}(\mathbf{x}) \mathrm{Prob}( (\mathbf{x},\sigma);t)  +  q \sum_{k=1}^N q^{l(\hat{\tau})-l(\hat{\sigma})}L_{mq\text{TAZRP}}(  (\mathbf{x}^{-k},\hat{\tau}),(\mbfx,\tau)) \cdot \mathrm{Prob}( (\mathbf{x}^{-k},\hat{\sigma});t) .
\end{align*}
And now applying the induction hypothesis shows that
$$
\frac{d}{dt} \mathrm{Prob}( (\mathbf{x},\sigma);t)  = -\mathrm{wt}(\mathbf{x}) \mathrm{Prob}( (\mathbf{x},\sigma);t)  +  q \sum_{k=1}^N L_{mq\text{TAZRP}}(  (\mathbf{x}^{-k},\hat{\tau}),(\mbfx,\tau)) \cdot \mathrm{Prob}( (\mathbf{x}^{-k},\hat{\tau});t) 
$$
At the same time, the master equation tells us directly that
$$
\frac{d}{dt} \left[ q \mathrm{Prob}( (\mathbf{x},\tau);t)\right]  = -\mathrm{wt}(\mathbf{x}) \cdot \left[ q \mathrm{Prob}( (\mathbf{x},\tau);t)  \right] +  q \sum_{k=1}^N L_{mq\text{TAZRP}}(  (\mathbf{x}^{-k},\hat{\tau}),(\mbfx,\tau)) \cdot \mathrm{Prob}( (\mathbf{x}^{-k},\hat{\tau});t) 
$$
Therefore $ \mathrm{Prob}( (\mathbf{x},\sigma);t) $ and $q \mathrm{Prob}( (\mathbf{x},\tau);t)$ satisfy the same differential equation, and because the initial condition is $q$--exchangeable, they must have the same value at $t=0$. Therefore they are equal for all values of $t$, completing the inductive step. 

By the first part of Theorem \ref{QBIN} and the Markov projection property, this reduces the theorem for $\mathbf{N}=(1,\ldots,1)$ to the single--species case, which is already known, thus completing the proof.

\subsubsection{Proof of Theorem \ref{lamelol}}
To show that this defines a stationary measure, it suffices to apply the differential operator which defines the master equation and show the result is $0$. 
 
Suppose $\mathrm{m}(\mathbf{x})=(m_1,\ldots,m_r)$. Suppose that the lattice site $x_{m_s}$, there are $m_s^{(j)}$ particles of species $j$. Then the stationary measures are given by
$$
\frac{q^{\inv(\sigma)}}{[m_1^{(1)}]_q^! \cdots [m_r^{(1)}]_q^! [m_1^{(2)}]_q^! \cdots [m_r^{(2)}]_q^! \cdots [m_1^{(n)}]_q^! \cdots [m_r^{(n)}]_q^! }
$$
Therefore, using \eqref{id} to expand $\mathbf{wt}(\mbfx)$, and using the same notation for the master equation as in the previous section, we want to show
\begin{multline*}
\sum_{k=1}^N q^{l(\hat{\sigma}) } L_{mq\text{TAZRP}}( (\mbfx^{-k},\hat{\sigma}),(\mbfx,\sigma)) \frac{[m_s^{(j)}]_q }{[m_1^{(1)}]_q^! \cdots [m_r^{(1)}]_q^! [m_1^{(2)}]_q^! \cdots [m_r^{(2)}]_q^! \cdots [m_1^{(n)}]_q^! \cdots [m_r^{(n)}]_q^! } \\
= q^{\inv(\sigma)}\frac{ \sum_{s=1}^r \sum_{j=1}^n q^{m_{s}^{(j+1)} + \cdots + m_{s}^{(n)}}[m_s^{(j)}]_q  }{[m_1^{(1)}]_q^! \cdots [m_r^{(1)}]_q^! [m_1^{(2)}]_q^! \cdots [m_r^{(2)}]_q^! \cdots [m_1^{(n)}]_q^! \cdots [m_r^{(n)}]_q^! } .
\end{multline*}
On the left--hand--side, the $(j,s)$ is related to $k$ in that the $k$th particle is of species $j$ located at lattice site $x_{m_1 + \ldots + m_s}$. The $[m_s^{(j)}]_q$ occurs in the numerator because the corresponding term in the denominator is replaced with $[m_s^{(j)}-1]_q^!$. By inserting the expression for the generator, it suffices to show that
$$
\inv(\hat{\sigma}) + m_{s+1}^{(1)} + \ldots + m_{s+1}^{(j-1)} = \inv(\sigma) + m_s^{(j+1)} + \ldots + m_s^{(n)}.
$$
However, this is straightforward from the definition of the dynamics.

\section{Multi-species ASEP}\label{3}

Let us begin by recalling a few results of the single--species ASEP.
\subsection{Single--species ASEP}
Let $\mathfrak{Z}\subseteq \mathbb{Z}$ be a finite or infinite subset of the set of integers. Given $X\in 2^{\mathfrak{Z}}$ (the subsets of $\mathfrak{Z}$) and $x\in X$, define
\begin{align*}
X^{x\rightarrow x+1} &= (X \backslash \{x\} ) \cup \{x+1\} , \text{ if } x+1 \in \mathfrak{Z}\backslash X, \\
X^{x\rightarrow x-1} &= (X \backslash \{x\} ) \cup \{x-1\} , \text{ if } x-1 \in \mathfrak{Z} \backslash X. \\
\end{align*}
In every other case, formally set $X^{x\rightarrow x\pm 1}$ to be undefined.

Define ASEP on $\mathfrak{Z}$ to be the continuous--time Markov process on state space $X\in 2^{\mathfrak{Z}}$ with generator
$$
L_{\text{ASEP}[\mathfrak{Z}]}(X,X')
= 
\begin{cases}
1, \text{ if } & X'=X^{x\rightarrow x+1} \text{ for some } x\in X, \\
q, \text{ if } & X'=X^{x \rightarrow x-1} \text{ for some } x\in X, \\
0, \text{ if } & X' \neq X, X^{x\rightarrow x\pm 1} \text{ for all } x\in X,
\end{cases}
$$
and let $L_{\text{ASEP}[\mathfrak{Z}]}(X,X)$ be the unique number such that every row of $L$ sums to $0$. The set $X$ describes the locations where the sites are occupied by particles.

When $\mathfrak{Z}=\mathbb{Z}$, there is an explicit formula for the transition probabilities of ASEP.
Set 
$$
S_{\alpha,\beta}= - \frac{1+q\xi_{\alpha}\xi_{\beta} - \xi_{\alpha}}{1+q\xi_{\alpha}\xi_{\beta} - \xi_{\beta}},
$$
where $\epsilon(\xi)=\xi^{-1} + q\xi -1$. Then for ASEP[$\mathbb{Z}$] with jumps to the right of rate $1$ and jumps to the left of rate $q$, the transition probabilities were found in \cite{TW}\footnote{In \cite{TW}, the formula is written where the probability of a right jump is $p$ and the probability of a left jump is $q$, where $p+q=1$. To match the notation, one simply replaces the $p$ with $1$ and rescales time by a factor of $1+q$.}:
$$
P^{\text{ASEP}}_Y(X;t) = \left( \frac{1}{2\pi i}\right)^N\sum_{\sigma \in S_N} \int_{\mathcal{C}_r} \cdots \int_{\mathcal{C}_r} A_{\sigma} \prod_{i} \xi_{\sigma(i)}^{x_i - y_{\sigma(i)}-1} e^{(1+q)\sum_i \epsilon(\xi_i)t} d\xi_1 \cdots d\xi_N,
$$
where the contour integrals are very small.

\subsection{Definition of multi--species ASEP}
Let us define the generator for the multi--species ASEP (mASEP) on $\mathfrak{Z}\subseteq \mathbb{Z}$. Additionally fix $\mathbf{N}=(N_1,\ldots,N_n)$. In this case, the state space will consist of pairs $(\mathbf{x},\sigma)$, where 
$$
\mathbf{x} \in \mathcal{W}_N^+ = \{(x_1 > x_2 > \ldots > x_N): x_i \in \mathfrak{Z}\} \subset \mathfrak{Z}^N
$$
and $\sigma \in D_{H'}^{-1}$, where $H'=S(\mathbf{N})$. By a slight abuse of notation, $\mathbf{x}$ will be equivalently considered as a subset of $\mathfrak{Z}$.
The generator is then defined by having off--diagonal entries
$$
L_{\text{mASEP}[\mathfrak{Z}]} (( \mathbf{x},\sigma),(\mathbf{x}',\sigma')) 
= 
\begin{cases}
1, \text{ if } \mathbf{x}' = \mathbf{x}^{x \rightarrow x+1} \text{ for some } x\in X \text{ and } \sigma=\sigma',\\
q, \text{ if } \mathbf{x}' = \mathbf{x}^{x \rightarrow x-1} \text{ for some } x\in X \text{ and } \sigma=\sigma',\\
1, \text{ if } \mathbf{x}' = \mathbf{x} \text{ and } \sigma' = \sigma \circ (r \ r+1) \text{ for some } r \text{ and } \inv(\sigma') = \inv(\sigma)-1,\\ 
q, \text{ if } \mathbf{x}' = \mathbf{x} \text{ and } \sigma' = \sigma \circ (r \ r+1) \text{ for some } r \text{ and } \inv(\sigma') = \inv(\sigma)+1,\\
0, \text{else}. 
\end{cases}
$$
The diagonal entries $L_{\text{mASEP}[\mathfrak{Z}]} (( \mathbf{x},\sigma),(\mathbf{x},\sigma)) $ are defined so that the rows sum to $0$. The set $\mathbf{x}$ indicates the locations where the sites are occupied by particles and $\sigma$ indicates the ordering of the particles. The multi--species ASEP also satisfies the property that the projection to the first $k$ species is again a multi--species ASEP.

The main theorem is:
\begin{theorem}\label{uhm...} Given $q$--exchangeable initial conditions supported at $\mathbf{y}$, the multi--species ASEP on $\mathbb{Z}$ satisfies
$$
\mathrm{Prob}((X,\sigma);t) = \left( \frac{1}{2\pi i}\right)^N \frac{q^{\inv(\sigma)} }{[N_1]_q^! \cdots [N_n]_q^!} \sum_{\sigma \in S_N} \int_{\mathcal{C}_r} \cdots \int_{\mathcal{C}_r} A_{\sigma} \prod_{i} \xi_{\sigma(i)}^{x_i - y_{\sigma(i)}-1} e^{(1+q)\sum_i \epsilon(\xi_i)t} d\xi_1 \cdots d\xi_N,
$$
\end{theorem}

\subsection{Proof of Theorem \ref{uhm...}}

\subsubsection{Generalities}
Suppose that $(X(t),S(t))$ is a Markov process on the state space $\mathcal{X} \times \mathcal{S}$, which we assume to be countably infinite or finite. In general $S(t)$ need not be Markov, but assume that $X(t)$ is a Markov process on $\mathcal{X}$. Let $L_X$ be the generator of $X(t)$. The generator $L_{XS}$ of $(X(t),S(t))$ can then be written as
$$
L_{XS}((x,s),(x',s')) = L_X(x,x')M_S^{(xx')}(s,s'),
$$
where $M_S^{(xx')}$ is some $\mathcal{S} \times \mathcal{S}$ matrix which depends on $x,x'\in \mathcal{X}$. More specifically, define
$$
M_S^{(xx')}(s,s')
= 
\begin{cases}
\displaystyle\frac{L_{XS}((x,s),(x',s'))}{L_X(x,x')}, \quad & L_X(x,x') \neq 0, \\
1_{\{s=s'\}}, \quad & L_X(x,x') = 0
\end{cases}
$$

\begin{proposition}\label{Proposition}
(a) If $L_X(x,x')\neq 0$, then the matrix elements of $M_S^{(xx')}$ satisfy
$$
M_S^{(xx')}(s,s') = 
\begin{cases}
\displaystyle\lim_{\epsilon\rightarrow 0^+} \mathbb{P}( S(t+\epsilon) = s' \ \vert X(t+\epsilon)=x' \text{ and }  (X(t),S(t))=(x,s)), \quad x\neq x' ,\\
\displaystyle\lim_{\epsilon\rightarrow 0^+} \epsilon^{-1}\frac{\mathbb{P}( S(t+\epsilon) = s' \ \vert X(t+\epsilon)=x \text{ and }  (X(t),S(t))=(x,s))}{L_X(x,x)}, \quad x= x', s \neq s',\\
1+\displaystyle\lim_{\epsilon\rightarrow 0^+} \epsilon^{-1}\frac{\mathbb{P}( S(t+\epsilon) = s \ \vert X(t+\epsilon)=x \text{ and }  (X(t),S(t))=(x,s))-1}{L_X(x,x)}, \quad x= x', s = s'
\end{cases}
$$
In particular, $(-1)^{1_{\{x = x'\}}} \left( M_S^{xx'}-\mathrm{Id} \right)$ is the generator of a continuous--time Markov process $S^{(xx')}(t)$ on $\mathcal{S}$ for all $x,x'\in \mathcal{X}$.

(b) Let $\nu$ be a probability measure on $\mathcal{S}$. Then $\nu M_S^{(xx')}=\nu$ if and only if $\nu$ is a stationary measure of the Markov process $S^{(xx')}(t)$. 

(c) Suppose that $\nu$ is a stationary measure of the Markov process $S^{(xx')}(t)$ for every $(x,x') \in \mathcal{X} \times \mathcal{X}$. Then for any probability measure $\mu$ on $\mathcal{X}$,
$$
(\mu \otimes \nu)L_{XS} = (\mu L_X) \otimes \nu.
$$

(d) Suppose that $\nu$ is a stationary measure of the Markov process $S^{(xx')}(t)$ for every $(x,x') \in \mathcal{X} \times \mathcal{X}$. Then for every $(x',s') \in \mathcal{X} \times \mathcal{S}$ and any $x\in \mathcal{X}$,
$$
\sum_{s\in \mathcal{S}} \nu(s) \mathbb{P}\left( (X(t),S(t)) = (x',s')\ \vert \ (X(0),S(0))=(x,s) \right)   = \nu(s') \mathbb{P}( X(t)=x' \vert X(0) = x).
$$
\end{proposition}
\begin{proof}
(a) 
Recall the well--known identity
$$
\mathbb{P}(A \vert B \cap C) = \frac{\mathbb{P}(A \cap B \vert C)}{\mathbb{P}(B \vert C)},
$$
which holds as long as $\mathbb{P}(B\cap C)\neq 0$.
%which follows from the definition of conditional probability, 
%$$
%\mathbb{P}( A \cap B \vert C) = \frac{\mathbb{P}(A \cap B \cap C)}{\mathbb{P}(C)} =  \frac{\mathbb{P}(A \cap B \cap C)}{\mathbb{P}(B \cap C) }\frac{\mathbb{P}(B \cap C) }{\mathbb{P}(C)} = \mathbb{P}(A \vert B \cap C) \mathbb{P}(B \vert C).
%$$
Letting $A,B,C$ be the events 
$$
A= \{ S(t+\epsilon)=s'\}, \quad B = \{ X(t+\epsilon)=x'\}, \quad C = \{(X(t),S(t))=(x,s)\},
$$
we see that $\mathbb{P}(B\cap C)\neq 0$ because $L_X(x,x')\neq 0$. By definition
$$
\mathbb{P}( (X(t+\epsilon),S(t+\epsilon)) = (x',s') \ \vert \  (X(t),S(t))=(x,s)) = 1_{\{x=x',s=s'\}} + L_X(x,x')M_S^{(xx')}(s,s')\epsilon+ O(\epsilon^2).
$$
Thus we have that (using the Markov property of $X(t)$)
$$
\mathbb{P}( S(t+\epsilon) = s' \ \vert X(t+\epsilon)=x' \text{ and }  (X(t),S(t))=(x,s))  = \frac{1_{\{x=x',s=s'\}} + L_X(x,x')M_S^{(xx')}(s,s')\epsilon+ O(\epsilon^2)}{ 1_{\{x=x'\}} + L_X(x,x')\epsilon + O(\epsilon^2)}.
$$
This immediately implies the first two cases. Setting $x=x',s=s'$ results in 
$$
\mathbb{P}( S(t+\epsilon) = s \ \vert X(t+\epsilon)=x \text{ and }  (X(t),S(t))=(x,s))  -1 = \frac{L_X(x,x)(M_S^{(xx)}(s,s)-1)\epsilon+ O(\epsilon^2)}{ 1 + L_X(x,x)\epsilon + O(\epsilon^2)},
$$
which implies the third case.

The last statement follows by noticing that each row sums to $1$ and that the off--diagonal entries are non--negative. 

(b) This follows immediately from the definition of $S^{(xx')}(t)$.

(c) We can compute that
\begin{align*}
\left[ (\mu \otimes \nu)L_{XS} \right](x',s') &= \sum_{(x,s) \in \mathcal{X} \times \mathcal{S}} \mu(x)\nu(s) L_{XS}((x,s),(x's')) \\
&= \sum_{x \in \mathcal{X}} \sum_{s \in \mathcal{S}} \mu(x) L_X(x,x') \nu(s) M_S^{(xx')}(s,s') \\
&= [\mu L_X](x') \nu(s').
\end{align*}

(d) The left--hand--side equals
$$
\left[(\delta_x \otimes \nu) e^{tL_{XS}}\right](x',s').
$$
By part (c), this then equals
$$
\left[ (\delta_x e^{tL_X}) \otimes \nu \right](x',s'),
$$
which equals the right--hand--side.
\end{proof}

\subsubsection{ASEP with distinct species}
The stationary measures of multi--species ASEP on a finite interval with closed boundary conditions were classified in Theorem 3.1 of \cite{BS3}, as well as blocking measures on the infinite lattice. See also the comment at the end of section 2 of \cite{ARI}.

In the special case when $\mathfrak{Z} = \{1,\ldots,N\}$ and $\mathbf{N}=(1,\ldots,1)$, so that every site is always occupied and there is only one particle of each species, the station measure has a simple expression. Namely, let $\nu_q$ be the probability measure on $S(N)$ defined by 
$$
\nu_q(\sigma) = \frac{q^{\inv(\sigma)}}{[N]_q^!}.
$$
Then for $\mathfrak{Z} = \{1,\ldots,N\}$,
$$
\nu_q L_{\text{mASEP}[\mathfrak{Z}] }= 0.
$$
In other words, $\nu_q$ is a stationary measure.

We consider a slightly more general case, where $\mathfrak{Z}\subset \mathbb{Z}$ consists of $k$ disjoint intervals, each of length $L_k$ and separated from each other by at least one lattice site. In this case, ASEP on $\mathfrak{Z}$ evolves as $k$ independent ASEPs, one on each interval. Suppose that $L_1 + \ldots + L_k=N$ and again $\mathbf{N}=(1,\ldots,1)$, so that again every lattice site is always occupied and there is only one particle of each species. We have an embedding of subgroups $S(L_1) \times \cdots \times S(L_k) \subseteq S(N)$, so $S(N)$ can be written as a disjoint union of cosets
$$
S(N) = \bigsqcup_{\tau \in S(N)/(S(L_1) \times \cdots \times S(L_k) ) } \tau ( S(L_1) \times \cdots \times S(L_k) ).
$$

\begin{figure}
\caption{This shows ASEP on $\mathfrak{Z}$, when $\mathfrak{Z} = \{1,2,5,6,7,8,9,11\}$.}
\begin{center}
\begin{tikzpicture}
\draw [thick](-0.5,0) -- (10.5,0) ;
\draw [black,fill=white] (0,0) circle (1ex);
\draw [black,fill=black] (1,0) circle (1ex);
\Large
\node at (2,0){$\times$};
\node at (3,0){$\times$};
\normalsize
\draw [black,fill=black] (4,0) circle (1ex);
\draw [black,fill=black] (5,0) circle (1ex);
\draw [black,fill=white] (6,0) circle (1ex);
\draw [black,fill=black] (7,0) circle (1ex);
\draw [black,fill=white] (8,0) circle (1ex);
\Large
\node at (9,0){$\times$};
\normalsize
\draw [black,fill=white] (10,0) circle (1ex);

\node at (5.8,0.25) (6-) {};
\node at (5.2,0.25) (5+) {};
\node at (5,0.25) (5) {};
\node at (5.8,0.25) (5-) {};
\node at (5.2,0.25) (4+) {};
\node at (6.8,0.25) (4-) {};
\node at (6.2,0.25) (3+) {};
\node at (7.8,0.25) (3-) {};
\node at (7.2,0.25) (2+) {};
\node at (0.8,0.25) (1-) {};
\node at (0.2,0.25) (0+) {};
\draw (1-) edge[out=90,in=90,->, line width=0.5pt] (0+);
\draw (4-) edge[out=90,in=90,->, line width=0.5pt] (3+);
\draw (2+) edge[out=90,in=90,->, line width=0.5pt] (3-);
\draw (4+) edge[out=90,in=90,->, line width=0.5pt] (5-);

\node at (0.5,0.8) {$q$};
\node at (6.5,0.7) {$q$};
\node at (7.5,0.8) {$1$};
\node at (5.5,0.8) {$1$};

\foreach \y in {1,2,3,4,5,6,7,8,9,10,11}{
     \node at (\y-1,-0.5) {\y};
  }

\end{tikzpicture}
\end{center}
\label{Blocks}
\end{figure}
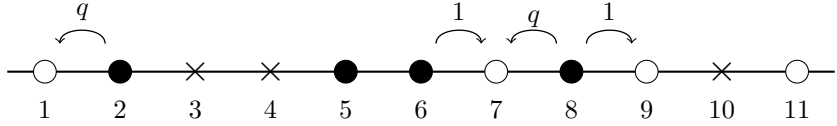

In this case, we have the following lemma:
\begin{lemma}\label{bbb}
Let $c(\cdot)$ be any probability measure on the set $S(N)/(S(L_1) \times \cdots \times S(L_k) )$. Then the probability measure on $S(N)$ defined by 
$$
\mathbb{P}( \tau(\sigma_1,\ldots,\sigma_k)) = c(\tau) \nu_q(\sigma_1) \cdots \nu_q(\sigma_k)
$$
is a stationary measure for ASEP on $\mathfrak{Z}$. In particular, the probability measure $\nu_q$ on $S(N)$ is a stationary measure. 
\end{lemma}
\begin{proof}
The first part of the lemma follows from the previous lemma.
because mASEP[$\mathfrak{Z}$] is independent copies of mASEP on each interval. 

For the second part of the lemma, we take 
$$
c(\tau) \propto q^{\inv(\tau)},
$$
where $\tau$ is the coset representative with the fewest inversions. By Lemma \ref{2.3.3}, this results in the probability measure $\nu_q$ on $S(N)$.

\end{proof}

Now we relate the multi--species ASEP to the framework of Proposition \ref{Proposition}. The role of $\mathcal{X}$ will be played by $\mathcal{W}_N^+$ and the role of $\mathcal{S}$ will be played by $D_{H'}^{-1}$, where $H'=S(\mathbf{N})$. Let $S^{(\mathbf{x}\mathbf{x}')}(t)$ be the Markov process on $D_{H'}^{-1}$.

\begin{lemma}\label{aaa}
Let $\mathbf{N}=(1,\ldots,1)$, and consider the Markov process $S^{(\mathbf{x}\mathbf{x})}(t)$ on $D_{H'}^{-1}=S(N)$. Then the Markov process $(\mathbf{x},S^{(\mathbf{x}\mathbf{x})}(t))$ is a mASEP on $\mathbf{x}$, with time rescaled by $-L_X(\mathbf{x},\mathbf{x})$.
\end{lemma}
\begin{proof}
This follows immediately from Proposition \ref{Proposition}(a) and the explicit expression for the generator of mASEP.
\end{proof}

\begin{proposition}
Assume $\mathbf{N}=(1,\ldots,1)$. Suppose that the initial condition $Y$ is $q$--exchangeable. Then for $\sigma,\sigma^0$ in $S_N$, and for all $t \geq 0$,
$$
 q^{ - \inv(\sigma)}   \cdot \mathrm{Prob}((\mathbf{x},\sigma);t) =  q^{-\inv(\sigma^0)}  \cdot \mathrm{Prob}((\mathbf{x},\sigma^0) ;t).
$$
\end{proposition}
\begin{proof}
The theorem follows immediately from Proposition \ref{Proposition}(c), Lemma \ref{bbb} and Lemma \ref{aaa}. Namely, the two lemmas tell us that the $q$--exchangeability of the initial conditions imply that the initial condition is of the form $(\mu \otimes \nu)$. Proposition \ref{Proposition}(c) then imply that the probability distribution is still of the same form at all times, which means that the distribution is still $q$--exchangeable. This is exactly the statement of the theorem.
\end{proof}

\subsubsection{The general case}

By Proposition \ref{Proposition}(d) and the result for the single--species ASEP, when $\mathbf{N}=(1,\ldots,1)$ the probability distribution is given by
$$
\mathrm{Prob}((\mathbf{x},\sigma^0);t) = \left( \frac{1}{2\pi i}\right)^N \frac{q^{\inv(\sigma^0)}}{[N]_q^!} \sum_{\tau \in S_N} \int_{\mathcal{C}_r} \cdots \int_{\mathcal{C}_r} A_{\tau} \prod_{i} \xi_{\tau(i)}^{x_i - y_{\tau(i)}-1} e^{(1+q)\sum_i \epsilon(\xi_i)t} d\xi_1 \cdots d\xi_N,
$$
where $\sigma^0 \in S(N)$. The more general case when $\mathbf{N}$ is arbitrary now follows from the Markov projection property and Proposition \ref{QBIN}: for $\sigma \in D_{H'}^{-1}$,

\begin{align*}
\mathrm{Prob}((\mathbf{x},\sigma);t) &= \sum_{\substack{\sigma^0 \in S(N) \\\sigma^0=a\sigma, a\in H'}} \mathrm{Prob}((\mathbf{x},\sigma^0);t) \\
&= \sum_{a \in H'} \left( \frac{1}{2\pi i}\right)^N \frac{q^{\inv(a)+\inv(\sigma)}}{[N]_q^!} \sum_{\tau \in S_N} \int_{\mathcal{C}_r} \cdots \int_{\mathcal{C}_r} A_{\tau} \prod_{i} \xi_{\tau(i)}^{x_i - y_{\tau(i)}-1} e^{(1+q)\sum_i \epsilon(\xi_i)t} d\xi_1 \cdots d\xi_N,
\end{align*}
yielding the theorem.

\bibliographystyle{plain}

\begin{figure}
\caption{The calculations used in Example 2.}
\begin{center}
\includegraphics{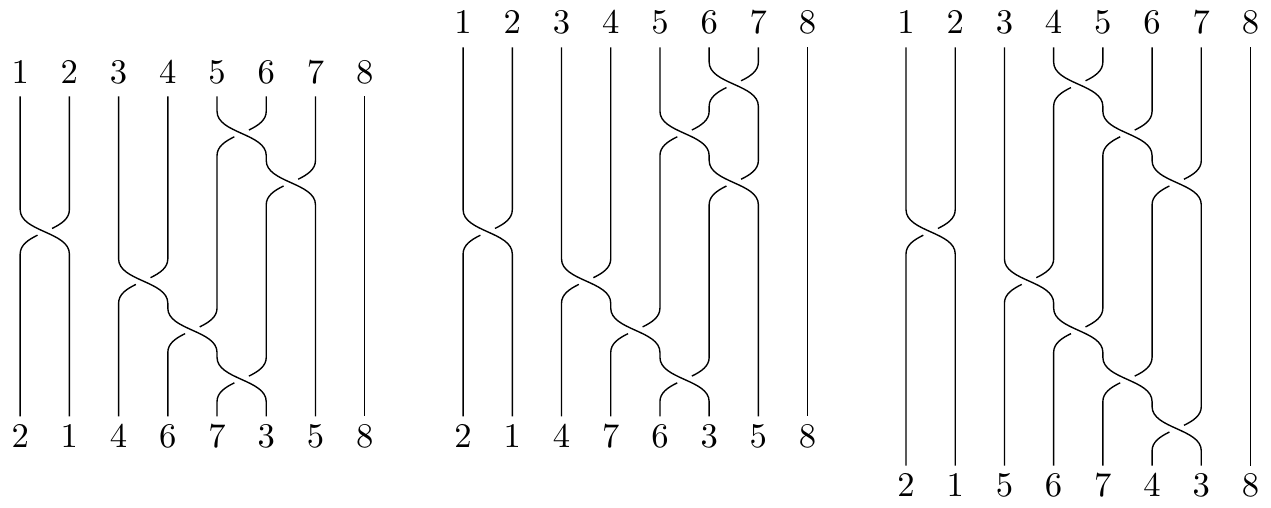}

\includegraphics{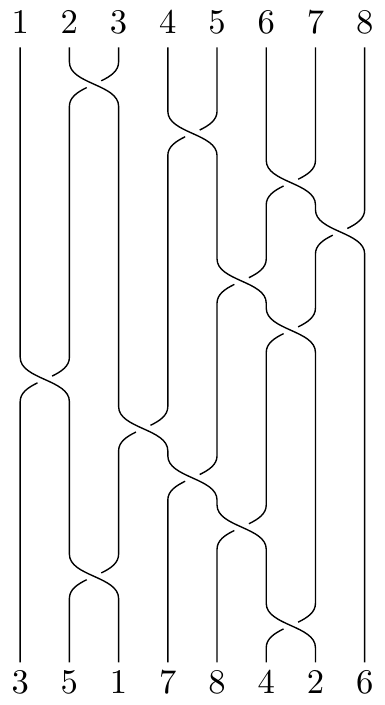}

\end{center}
\end{figure}

\end{document}